\documentclass[12pt,letterpaper]{article}
\usepackage{amsmath,amsfonts,amsthm,amssymb}

\swapnumbers 

\newtheorem{lemma}{Lemma}[section]
\newtheorem{corollary}[lemma]{Corollary}
\newtheorem{theorem}[lemma]{Theorem}
\newtheorem{assumptions}[lemma]{Standing Assumptions}

\theoremstyle{definition} 
\newtheorem{definition}[lemma]{Definition}
\newtheorem{remark}[lemma]{Remark}

\newtheorem{remarks}[lemma]{Remarks}

\newcommand{\Nat}{{\mathbb N}}

\newcommand\reals{{\mathbb R}}

\newcommand{\dg}{\sp{\text{\rm o}}}

\newcommand{\spec}{\operatorname{spec}}

\begin{document}

\title{A Loomis-Sikorski theorem and functional calculus for a
generalized Hermitian algebra}

\author{David J. Foulis{\footnote{Emeritus Professor, Department of
Mathematics and Statistics, University of Massachusetts, Amherst,
MA; Postal Address: 1 Sutton Court, Amherst, MA 01002, USA;
foulis@math.umass.edu.}}\hspace{.05 in} Anna Jen\v cov\'a and Sylvia
Pulmannov\'{a}{\footnote{ Mathematical Institute, Slovak Academy of
Sciences, \v Stef\'anikova 49, SK-814 73 Bratislava, Slovakia; jenca@mat.savba.sk,
pulmann@mat.savba.sk. The second and third authors were supported by grant
VEGA No.2/0069/16.}}}

\date{}

\maketitle

\begin{abstract}
A generalized Hermitian (GH-) algebra is a generalization of the partially
ordered Jordan algebra of all Hermitian operators on a\linebreak Hilbert
space. We introduce the notion of a gh-tribe, which is a commutative
GH-algebra of functions on a nonempty set $X$ with pointwise partial
order and operations, and we prove that every commutative GH-algebra is
the image of a gh-tribe under a surjective GH-morphism. Using this result,
we prove each element $a$ of a GH-algebra $A$ corresponds to a real observable
$\xi\sb{a}$ on the $\sigma$-orthomodular lattice of projections in $A$ and
that $\xi\sb{a}$ determines the spectral resolution of $a$. Also, if $f$
is a continuous function defined on the spectrum of $a$, we formulate
a definition of $f(a)$, thus obtaining a continuous functional calculus
for $A$.
\end{abstract}

\section{Introduction}

Generalized Hermitian (GH-) algebras, which were introduced in \cite{FPgh}
and further studied in \cite{FPspin, FPreg}, incorporate several important
algebraic and order theoretic structures including effect algebras \cite{FoBe},
MV-algebras \cite{Chang}, orthomodular lattices \cite{K}, Boolean algebras
\cite{Sik}, and Jordan algebras \cite{Mcc}. Apart from their intrinsic
interest, all of the latter structures host mathematical models for
quantum-mechanical notions such as observables, states, properties, and
experimentally testable propositions \cite{DvPu, Var} and thus are pertinent
in regard to the quantum-mechanical theory of measurement \cite{BLPM}.

It turns out that GH-algebras are special cases of the more general synaptic
algebras introduced in \cite{Fsa}, and further studied in \cite{FPproj, FPtype,
FPSym, FPcom, FJP2proj, FJPprojef, FJPstat, FJPvect, Pid}. Thus, in this paper,
it will be convenient for us to treat GH-algebras as special kinds of synaptic
algebras (see Section \ref{sc:SAGHA} below).

A real observable $\xi$ for a physical system ${\mathcal S}$ is understood to
be a quantity that can be experimentally measured, and that when mesured yields
a result in a specified set $\reals\sb{\xi}$ of real numbers. If $f$ is a
function defined on $\reals\sb{\xi}$, then $f(\xi)$ is defined to be the
observable that is measured by measuring $\xi$ to obtain, say, the result
$\lambda\in\reals\sb{\xi}$, and then regarding the result of this measurement
of $f(\xi)$ to be $f(\lambda)$. A state $\rho$ for ${\mathcal S}$ assigns to
$\xi$ an expectation, i.e., the long-run average value of a sequence of
independent measurements of $\xi$ in state $\rho$.

As indicated by the title, one of our purposes in this paper is to formulate and
prove a version of the Loomis-Sikorski theorem for commutative GH-algebras
(Theorem \ref{th:LSCGH} below). This theorem is a generalization of the classical
Loomis-Sikorski theorem for Boolean $\sigma$-algebras. In Theorem \ref{th:observ}
we use our generalized Loomis-Sikorski theorem to show that each element $a$ in a
GH-algebra $A$ corresponds to a real observable $\xi\sb{a}$. In Corollary \ref
{co:stat}, we obtain an integral formula for the expectation of the observable
$\xi\sb{a}$ in state $\rho$. Definition \ref{df:f(a)} and the following results
provide a continuous functional calculus for $A$.

\section{Preliminaries} \label{sc:Prelim}

In this section we review some notions and some facts that will be needed
as we proceed. We abbreviate `if and only if' as `iff,' the notation $:=$
means `equals by definition,' $\reals$ is the ordered field of real numbers,
$\reals\sp{+}:=\{\alpha\in\reals:0\leq\alpha\}$, and $\Nat:=\{1,2,3,...\}$
is the well-ordered set of natural numbers.

\begin{definition} \label{df:CompletenessProps}
Let ${\mathcal P}$ be a partially ordered set (poset). Then{\rm:}
\begin{enumerate}
\item[(1)] Let $p,q\in{\mathcal P}$. Then an existing supremum, i.e. least upper bound,
 (an existing infimum, i.e., greatest lower bound) of $p$ and $q$ in ${\mathcal P}$
 is written as $p\vee q$ ($p\wedge q$). If it is necessary to make clear that the
 supremum (infimum) is calculated in ${\mathcal P}$, we write $p\vee\sb{\mathcal P}q$
 ($p\wedge\sb{\mathcal P}q$).  $\mathcal P$ is a \emph{lattice} iff $p\vee q$ and
 $p\wedge q$ exist for all $p,q\in{\mathcal P}$. If ${\mathcal P}$ is a lattice,
 then a nonempty subset ${\mathcal Q}\subseteq{\mathcal P}$ is a \emph{sublattice}
 of ${\mathcal P}$ iff, for all $p,q\in{\mathcal Q}$, $p\vee q, p\wedge q\in
 {\mathcal Q}$, in which case ${\mathcal Q}$ is a lattice in its own right with
 $p\vee\sb{\mathcal Q}q=p\vee q$ and $p\wedge\sb{\mathcal Q}q=p\wedge q$.
\item[(2)] A lattice ${\mathcal P}$ is \emph{distributive} iff, for all $p,q,r\in
 {\mathcal P}$, $p\wedge(q\vee r)=(p\wedge q)\vee(p\wedge r)$, or equivalently,
 $p\vee(q\wedge r)=(p\vee q)\wedge(p\vee r)$.
\item[(3)] The poset ${\mathcal P}$ is \emph{bounded} iff there are elements,
 usually denoted by $0$ and $1$, such that $0\leq p\leq 1$ for all $p\in
 {\mathcal P}$. If ${\mathcal P}$ is a bounded lattice then elements $p,q\in
 {\mathcal P}$ are \emph{complements} of each other iff $p\wedge q=0$  and
 $p\vee q=1$. A \emph{Boolean algebra} is a bounded distributive lattice
 in which every element has a complement. In a Boolean algebra ${\mathcal P}$,
 the complement of an element $p$ is unique, and is often denoted by $p\sp{\prime}$.
\item[(3)] ${\mathcal P}$ is \emph{$\sigma$-complete} iff every sequence in
 $\mathcal P$ has both a supremum and an infimum in ${\mathcal P}$.
\item[(4)] ${\mathcal P}$ is \emph{Dedekind $\sigma$-complete} iff every sequence
 in ${\mathcal P}$ that is bounded above (below) has a supremum (an infimum) in
 ${\mathcal P}$.
\item[(5)] If $p\sb{1}\leq p\sb{2}\leq p\sb{3}\leq\cdots$ is an ascending sequence
 in ${\mathcal P}$ with supremum $p=\bigvee\sb{n=1}\sp{\infty}p\sb{n}$ in
 ${\mathcal P}$, we write $p\sb{n}\nearrow p$. Similar notation $p\sb{n}\searrow p$
 applies to a descending sequence $p\sb{1}\geq p\sb{2}\geq p\sb{3}\geq\cdots$ with
 infimum $p=\bigwedge\sb{n=1}\sp{\infty}p\sb{n}$ in ${\mathcal P}$.
\item[(6)] ${\mathcal P}$ is \emph{monotone $\sigma$-complete} iff, for every
 bounded ascending (descending) sequence $(p\sb{n})\sb{n=1}\sp{\infty}$ in
 ${\mathcal P}$, there exists $p\in G$ with $p\sb{n}\nearrow p$ ($p\sb{n}\searrow p$).
\end{enumerate}
\end{definition}

\begin{remarks} \label{rm:involution}
An \emph{involution} on the poset ${\mathcal P}$ is a mapping $\sp{\prime}\colon
{\mathcal P}\to{\mathcal P}$ such that, for all $p,q\in{\mathcal P}$, $p\leq q
\Rightarrow q\sp{\prime}\leq p\sp{\prime}$ and $(p\sp{\prime})\sp{\prime}=p$.
For instance, the complementation mapping $p\mapsto p\sp{\prime}$ on a Boolean
algebra ${\mathcal P}$ is an involution. An involution on  a poset ${\mathcal P}$
provides a ``duality" between existing suprema and infima of subsets ${\mathcal Q}$
of ${\mathcal P}$ as follows: If the supremum $\bigvee{\mathcal Q}$ (the infimum
$\bigwedge{\mathcal Q}$) exists, then the infimum $\bigwedge\{q\sp{\prime}:q\in
{\mathcal Q}\}$ (the supremum $\bigvee\{q\sp{\prime}:q\in{\mathcal Q}\}$) exists
and equals $(\bigvee{\mathcal Q})\sp{\prime}$ (equals $(\bigwedge{\mathcal Q})
\sp{\prime}$). Thus, if ${\mathcal P}$ admits an involution, then condition (3)
in Definition \ref{df:CompletenessProps} is equivalent to the same condition for
sequences bounded above (below) only, and similar remarks hold for conditions (4)
and (6).
\end{remarks}

Recall that an \emph{order unit normed space} $(V,u)$ \cite[pp. 67--69]{Alf}
is a partially ordered linear space $V$ over $\reals$ with positive cone
$V\sp{+}=\{v\in V:0\leq v\}$ such that: (1) $V$ is \emph{archimedean}, i.e.,
if $v\in V$ and $\{nv:n\in\Nat\}$ is bounded above in $V$, then $-v\in V\sp{+}$.
(2) $u\in V\sp{+}$ is an \emph{order unit} (sometimes called a \emph{strong order
unit} \cite{Sch}), i.e., for each $v\in V$ there exists $n\in\Nat$ such that $v
\leq nu$. Then the \emph{order-unit norm} on $(V,u)$ is defined by $\|v\|=\inf
\{\lambda\in{\mathbb R}^+: -\lambda u\leq v\leq\lambda u\}$ for all $v\in V$. If
$(V,u)$ is an order unit normed space, then the mapping $v\mapsto -v$ is an
involution on $V$, so Remarks \ref{rm:involution} apply. If $V$ is an archimedean
partially ordered real linear space and $u\in V\sp{+}$ is an order unit, then---if
it is understood $u$ is the order unit in question---we may simply say that $V$
(rather than $(V,u)$) is an order unit normed space.

See \cite[Proposition 2.I.2]{Alf} and \cite[Proposition 7.12.(c)]{Good}, for a
proof of the following.

\begin{lemma} \label{lm:NormProps}
If $(V,u)$ is an order unit normed space, then for all $v,w\in V$, {\rm(i)}
$-\|v\|u\leq v\leq\|v\|u$ and {\rm(ii)} $-w\leq v\leq w \Rightarrow\|v\|\leq\|w\|$.
\end{lemma}

The next theorem (see \cite{FPHand}) follows, \emph{mutatis mutandis}, from David
\linebreak Handelman's proof of \cite[Proposition 3.9]{Hand}.

\begin{theorem} \label{th:Handelman}
If $(V,u)$ is an order unit normed space such that $V$ is monotone
$\sigma$-complete, then $V$ is a Banach space under the order-unit norm.
\end{theorem}

Let $(V,u)$ be an order unit normed space. Then elements of the ``unit interval"
$V[0,u]:=\{e\in V:0\leq e\leq u\}$ are called \emph{effects}, and $V[0,u]$ is
organized into a so-called \emph{effect algebra} $(V[0,u];0,u,\sp{\perp},\oplus)$
\cite{FoBe} as follows: For $e,f\in V[0,u]$, $e\oplus f$ is defined iff $e+f\leq u$,
and then $e\oplus f:=e+f$; moreover, $e\sp{\perp}:=u-e$. The effect algebra $V[0,u]$
is partially ordered by the relation $e\leq f$ iff there is $g\in V[0,u]$ such that
$e\oplus g=f$, and this ordering coincides with that inherited from $V$. The mapping
$e\mapsto e\sp{\perp}$ is an involution on $V[0,u]$, so Remarks \ref{rm:involution}
apply. With the convexity structure induced by the linearity of $V$, $V[0,u]$ is a
convex effect algebra \cite{BBGP, BGP}. Two effects $e,f\in V[0,u]$ are said to be
\emph{{\rm(}Mackey{\rm)} compatible} iff there are elements $e\sb{1}, f\sb{1}, d\in
V[0,u]$ such that $e\sb{1}+f\sb{1}+d\leq u$, $e=e\sb{1}+d$, and $f=f\sb{1}+d$.

An effect algebra that forms a lattice is said to be \emph{lattice ordered},
or simply a \emph{lattice effect algebra} \cite{ZR}. A lattice effect algebra
in which every pair of elements are compatible is called an \emph{MV-effect
algebra}. As a lattice, an MV-effect algebra is distributive. It is known that
MV-effect algebras are mathematically equivalent to MV-algebras, which were
introduced by Chang \cite{Chang} as algebraic bases for many-valued logics.
(See, e.g., \cite{DvPu} for the relations between lattice effect algebras and
MV-algebras). Notice that a convex effect algebra is an MV-effect algebra iff
it is a lattice \cite{BBGP}. By a well-known result of D. Mundici \cite{Munint},
every MV-algebra is isomorphic to the unit interval $G[0,u]$ in a lattice ordered
abelian group (abelian $\ell$-group) $G$ with order unit $u$. An MV-algebra that
is a $\sigma$-complete lattice is called a $\sigma$\emph{MV-algebra} \cite{Dvls},
and it turns out that an MV-algebra is a $\sigma$MV-algebra iff the corresponding
abelian $\ell$-group is Dedekind $\sigma$-complete.

A partially ordered linear space $V$ over $\reals$ that is a lattice under
the partial order is called a \emph{vector lattice} or a \emph{Riesz space}.
Every vector lattice, and more generally, every abelian $\ell$-group, is
distributive \cite[Theorem 4]{Birk}. If $V$ is a vector lattice and $v\in V$,
then the \emph{absolute value} and the \emph{positive part} of $v$ are denoted
and defined by $|v|:=v\vee(-v)$ and $v\sp{+}:=v\vee 0$, respectively. A vector
lattice $V$ satisfies \emph{Dedekind's law}: For $v,w\in V$, $v\vee w + v\wedge
w=v+w$.

\begin{lemma} \label{lm:VecLat}
If $V$ is a vector lattice, then $V$ is monotone $\sigma$-complete iff
$V$ is Dedekind $\sigma$-complete.
\end{lemma}
\begin{proof}
Suppose that $V$ is monotone $\sigma$-complete, let $(v\sb{n})\sb{n\in\Nat}$
be a sequence in $V$ that is bounded above, and define $(w\sb{n})\sb{n\in\Nat}$
by $w\sb{n}:=v\sb{1}\vee v\sb{2}\vee\cdots\vee v\sb{n}$ for all $n\in\Nat$. Then,
$(w\sb{n})\sb{n\in\Nat}$ is an ascending sequence in $V$ with the same set of
upper bounds as $(v\sb{n})\sb{n\in\Nat}$, whence $w\sb{n}\nearrow w\in V$, and
$w=\bigvee\sb{n\in\Nat}v\sb{n}$. Thus $V$ is Dedekind $\sigma$-complete. The
converse is obvious.
\end{proof}

Of course, by an \emph{order unit normed vector lattice}, we mean an order
unit normed space that is also a vector lattice.

\begin{theorem} \label{th:OUNVL}
If $V$ is a monotone $\sigma$-complete vector lattice with order unit $u$,
then $(V,u)$ is an order unit normed vector lattice and a Banach space
under the order-unit norm.
\end{theorem}

\begin{proof}
By \cite[3.2.5 Prop. 2]{Crist}, every monotone $\sigma$-complete vector
lattice is archimedean. By Theorem \ref{th:Handelman}, every monotone
$\sigma$ complete order unit normed space is Banach.
\end{proof}

\begin{lemma} \label{lm:rhovsnorm}
Let $(V,u)$ be an order unit normed vector lattice and let $v,w\in V$.
Then{\rm: (i)} $0<\epsilon\in\reals \Rightarrow |w|\leq\epsilon u
\Leftrightarrow\|w\|\leq\epsilon$. {\rm(ii)} If $(v\sb{n})\sb{n\in\Nat}$
is a sequence in $V$, then $\lim\sb{n\rightarrow\infty}v\sb{n}=v$ iff,
for every $0<\epsilon\in\reals$, there exists $N\in\Nat$ such that, for
all $n\in\Nat$, $N\leq n\Rightarrow|v\sb{n}-v|\leq\epsilon u$.
\end{lemma}

\begin{proof}
Let $0<\epsilon\in\reals$ and $v,w\in V$. (i) By Lemma \ref{lm:NormProps}
(i), $w,-w\leq\|w\|u$, whence $\|w\|\leq\epsilon\Rightarrow w,-w\leq\epsilon u
\Rightarrow -\epsilon u\leq w\leq\epsilon u$. Conversely, $-\epsilon u\leq w\leq
\epsilon u\Rightarrow\|w\|\leq\epsilon$, whence $-\epsilon u\leq w\leq\epsilon u
\Leftrightarrow\|w\|\leq\epsilon$. Thus, since $|w|=w\vee(-w)$, we have
$|w|\leq\epsilon u\Leftrightarrow w,-w\leq\epsilon u\Leftrightarrow-\epsilon u
\leq w\leq\epsilon u\Leftrightarrow\|w\|\leq\epsilon$.  (ii) Putting $w=
v\sb{n}-v$ in (i), we have $|v\sb{n}-v|\leq\epsilon u\Leftrightarrow\|v\sb{n}-v\|
\leq\epsilon$, from which (ii) follows.
\end{proof}

Evidently, if $(V,u)$ is a Dedekind $\sigma$-complete order unit normed vector
lattice, then the unit interval $V[0,u]$ is a convex $\sigma$MV-algebra. And
conversely, by \cite{BBGP, Munint}, every convex $\sigma$MV-algebra is
isomorphic to the unit interval in a Dedekind $\sigma$-complete order unit normed
vector lattice.

Let $(V,u)$ be an order unit normed vector lattice. An element $e\in V$ is
called a \emph{characteristic element} \cite[Definition p. 127]{Good} (or a
\emph{unitary element} \cite[Definition 4.1.1.2.]{Crist}) iff $e\wedge(u-e)=0$.
The set $B$ of all characteristic elements in $V$ is a sub-effect algebra of
$V[0,u]$; moreover, $B$ is a sublattice of $V$, and as such $B$ is a Boolean
algebra with $u-e$ as the complement of $e\in B$ \cite[Remark p. 120]{Crist}.
Two elements $e,f\in B$ are said to be \emph{orthogonal} iff $e\wedge f=0$,
or equivalently, iff $e\leq u-f$. By Dedekind's law, if $e$ and $f$ are
orthogonal elements of $V$, then $e\vee f=e+f$. An element in $V$ is said
to be \emph{simple} iff it is a finite real linear combination of characteristic
elements. By an adaptation of the proof of \cite[Theorem 5.1]{FPsr}, it can be
shown that every simple element in $V$ can be written as a finite real linear
combination of pairwise orthogonal characteristic elements.

Let $X$ be a compact Hausdorff space. We define ${\mathcal F}(X)$ to be the field of
all compact open (clopen) subsets of $X$, noting that ${\mathcal F}(X)$ is a Boolean
algebra under set containment $\subseteq$.  Also, as usual, $C(X,\reals)$ denotes
the partially ordered commutative associative real unital Banach algebra of all
continuous functions $f\colon X\to\reals$, with pointwise partial order and pointwise
finitary operations. Then, with the constant function $x\mapsto 1$ (denoted simply
by $1$) as order unit, $C(X,\reals)$ is an order unit normed vector lattice, and the
order-unit norm coincides with the supremum (or uniform) norm on $C(X,\reals)$. For
$f\in C(X,\reals)$, $|f|=f\vee(-f)$ is the pointwise absolute value, and it follows
from Lemma \ref{lm:rhovsnorm} (ii) that order-unit norm limits of sequences in
$C(X,\reals)$ are pointwise limits.

We denote by $P(X,\reals)$ the Boolean algebra of all characteristic elements in
$C(X,\reals)$. It is not difficult to see that $P(X,\reals)=\{p\in C(X,\reals):p=
p\sp{2}\}$ and also that $P(X,\reals)$ consists of all characteristic set functions
(indicator functions) $\chi\sb{K}$ of clopen sets $K\in{\mathcal F}(X)$, whence
the Boolean algebra $P(X,\reals)$ is isomorphic to ${\mathcal F}(X)$ under the
mapping $\chi_K \mapsto K$.

A \emph{Stone space} is a compact Hausdorff space $X$ such that the clopen sets
in $\mathcal{F}(X)$ form a basis for the open sets in $X$. If $X$ is a Stone
space, and $x,y\in X$ with $x\not=y$, then there are disjoint clopen sets $K,L
\in{\mathcal F}(X)$ with $x\in K$ and $y\in L$; hence $\chi\sb{K}(x)=1$ and
$\chi\sb{X\setminus L}(y)=0$. Therefore distinct points in $X$ are separated
by continuous functions in $P(X,\reals)\subseteq C(X,\reals)$.

By the \emph{Stone representation theorem}, if ${\mathcal P}$ is a Boolean algebra,
there is a Stone space $X$, called \emph{the Stone space of ${\mathcal P}$} and
uniquely determined up to a homeomorphism, such that ${\mathcal P}$ is isomorphic
to the Boolean algebra ${\mathcal F}(X)$, whence ${\mathcal P}$ is isomorphic to $P(X,
\reals)$. As is well-known, ${\mathcal P}$ is a $\sigma$-complete Boolean algebra
(a Boolean $\sigma$-algebra) iff the Stone space $X$ of ${\mathcal P}$ is basically
disconnected, i.e., the closure of every open F$\sb{\sigma}$ subset of $X$ remains
open.

By \cite[Theorem 2, p. 150]{Crist} and Theorem \ref{lm:VecLat}, we have the following.

\begin{theorem} \label{th:cristTh2}
Let $(V,u)$ be a monotone $\sigma$-complete order unit normed vector lattice, let
$B$ be the Boolean $\sigma$-algebra of characteristic elements in $V$, and let
$X$ be the basically disconnected Stone space of $B$. Then{\rm: (i)} $V$ is Dedekind
$\sigma$-complete. {\rm(ii)} There is an isomorphism of order unit normed vector
lattices, $\Psi\colon V\to C(X,\reals)$, of $V$ onto $C(X,\reals)$ such that the
restriction $\psi$ of $\Psi$ to $B$ is a Boolean isomorphism of $B$ onto $P(X,
\reals)$ as per Stone's theorem.
\end{theorem}

\section{Synaptic algebras and generalized Hermitian algebras} \label{sc:SAGHA}

The axioms SA1--SA8 for a synaptic algebra appear in the following
definition \cite[Definition 1.1]{Fsa}. To help fix ideas before we proceed,
we remark that the system $A={\mathcal{B}}\sp{sa}({\mathcal{\frak H}})$ of all
bounded self-adjoint linear operators on a Hilbert space ${\frak H}$ with the
algebra $R={\mathcal{B}}({\mathcal{\frak H}})$ of all bounded linear operators
on ${\frak H}$ as its enveloping algebra is a prototypic example of a
synaptic algebra.

\begin{definition} \label{df:SynapticAlgebra}
Let $R$ be a linear associative algebra over the real or complex numbers
with unity element $1$ and let $A$ be a real vector subspace of $R$.
If $a,b\in A$, we understand that the product $ab$, which may or may not
belong to $A$, is calculated in $R$, and we say that $a$ \emph{commutes
with} $b$, in symbols $aCb$, iff $ab=ba$. If $a\in A$ and $B\subseteq A$,
we define $C(a):=\{b\in A:aCb\},\ C(B):=\bigcap\sb{b\in B}C(b),\ CC(B):=
C(C(B)),\text{\ and\ } CC(a):=C(C(a))$. Of course, $B$ is said to be
\emph{commutative} iff $aCb$ for all $a,b\in B$, i.e., iff $B\subseteq C(B)$.

The vector space $A$ is a \emph{synaptic algebra} with \emph{enveloping
algebra} $R$ iff the following conditions are satisfied:
\begin{enumerate}
\item[SA1.] $(A,1)$ is an order unit normed space with positive cone $A\sp{+}
 :=\{a\in A:0\leq a\}$ and $\|\cdot\|$ is the corresponding order-unit norm
 on $A$.
\item[SA2.] If $a\in A$ then $a\sp{2}\in A\sp{+}$.
\item[SA3.] If $a,b\in A\sp{+}$, then $aba\in A\sp{+}$.
\item[SA4.] If $a\in A$ and $b\in A\sp{+}$, then $aba=0\Rightarrow
 ab=ba=0$.
\item[SA5.] If $a\in A\sp{+}$, there exists $b\in A\sp{+}\cap CC(a)$
 such that $b\sp{2}=a$.
\item[SA6.] If $a\in A$, there exists $p\in A$ such that $p=p\sp{2}$ and,
 for all $b\in A$, $ab=0\Leftrightarrow pb=0$.
\item[SA7.] If $1\leq a\in A$, there exists $b\in A$ such that $ab=ba=1$.
\item[SA8.] If $a,b\in A$, $a\sb{1}\leq a\sb{2}\leq a\sb{3}\leq\cdots$
 is an  ascending sequence of pairwise commuting elements of $C(b)$
 and $\lim\sb{n\rightarrow\infty}\|a-a\sb{n}\|=0$, then $a\in C(b)$.
\end{enumerate}
\end{definition}

For the remainder of this paper \emph{we assume that $A$ is a synaptic algebra}.
We shall also assume that $1\not=0$ (i.e., $A\not=\{0\}$), which enables us (as
usual) to identify each real number $\lambda\in\reals$ with $\lambda1\in A$. Thus,
for $a\in A$, $\|a\|=\inf\{0<\lambda\in\reals:-\lambda\leq a\leq\lambda\}$.
Limits in $A$ are understood to be limits with respect to the order-unit norm
$\|\cdot\|$. In the remainder of this section we comment briefly on some of the
basic consequences of SA1--SA8. See \cite{Fsa} for proofs and more details.

If $a\in A$, then by SA2, $a\sp{2}\in A\sp{+}\subseteq A$. Consequently, $A$ is
a special Jordan algebra with the Jordan product
\[
a\odot b:=\frac12(ab+ba)=\frac14[(a+b)\sp{2}-(a-b)\sp{2}]\in A\text{\ for all\ }
 a,b\in A.
\]
Let $a,b\in A$. Then $\|a\odot b\|\leq\|a\|\|b\|$. Also, if $aCb$, then $ab=ba
=a\odot b\in A$. Thus, since $a\sp{2}\in A$ and $a\sp{2}Ca$, we have $a\sp{3}
=a\sp{2}a\in A$, and by induction, $a\sp{n}\in A$ for all $n\in N$. Therefore,
if $q(t)=\alpha\sb{0}+\alpha\sb{1}t+\alpha\sb{2}t\sp{2}+\cdots+\alpha\sb{n}
t\sp{n}$ is a real polynomial, then $q(a):=\alpha\sb{0}+\alpha\sb{1}a+\alpha
\sb{2}a\sp{2}+\cdots+\alpha\sb{n}a\sp{n}\in A$, i.e., $A$ is closed under the
formation of real polynomials in its elements.

By a simple calculation, $aba=2a\odot(a\odot b)-a\sp{2}\odot b\in A$, and
the mapping $b\mapsto aba$, called the \emph{quadratic mapping} determined by
$a$, turns out to be linear and order preserving. In particular, $A$ satisfies
the condition $a\in A$, $b\in A\sp{+}\Rightarrow aba\in A\sp{+}$, which is
stronger than SA3. By SA4 with $b=1$, we have $a\sp{2}=0\Rightarrow a=0$.

Using SA5 and SA2, one can prove that, if $a,b\in A\sp{+}$ and $aCb$, then $ab
\in A\sp{+}$ \cite[Lemma 1.5]{Fsa}. Also, using SA5, and arguing as in \cite
[Theorem 2.2]{Fsa}, it can be shown that every $a\in A\sp{+}$, has a \emph
{square root} $a\sp{1/2}\in A\sp{+}$ which is uniquely determined by the condition
$(a\sp{1/2})\sp{2}=a$; moreover, $a\sp{1/2}\in CC(a)$.

Let $a\in A$. Then $a\sp{2}\in A\sp{+}$ by SA2, and the \emph{absolute
value} of $a$, defined by $|a|:=(a\sp{2})\sp{1/2}$, has the property that
$|a|\in CC(a)$. Moreover, $|a|$ is uniquely determined by the properties
$|a|\in A\sp{+}$ and $|a|\sp{2}=a\sp{2}$. Furthermore, using the
absolute value of $a$, we define the \emph{positive part} and the \emph
{negative part} of $a$ by $a\sp{+}:=\frac12(|a|+a)\in A\sp{+}\cap CC(a)$
and $a\sp{-}:=(-a)\sp{+}=\frac12(|a|-a)\in A\sp{+}\cap CC(a)$, respectively.
Then $a=a\sp{+}-a\sp{-}$, $|a|=a\sp{+}+a\sp{-}$, $a\sp{+}a\sp{-}=0=a\sp{-}
a\sp{+}$. (See Lemma \ref{lm:simple,etc} below.)

As in Section \ref{sc:Prelim}, we define $E:=A[0,1]=\{e\in A:0\leq e\leq
1\}$ and organize $E$ into a convex effect algebra. An important subset of
$E$ is the set $P:=\{p\in A: p\sp{2}=p\}$ of idempotents in $A$, which are
called \emph{projections}. By \cite[Theorem 2.6 (iii) and (v)]{Fsa} we have
the following.

\begin{lemma} \label{lm:Projections}
If $p\in E$, then the following conditions are mutually equivalent{\rm: (i)}
$p\in P$. {\rm(ii)} $p$ is an extreme point of $E$. {\rm(iii)} $p\wedge
\sb{E}(1-p)=0$.
\end{lemma}

With the partial order inherited from $A$ and the orthocomplementation
$p\mapsto p\sp{\perp}:=1-p$, $P$ is an orthomodular lattice (OML) \cite{K}
with smallest element $0$ and largest element $1$. Two projections $p,q\in P$
are said to be \emph{orthogonal}, in symbols $p\perp q$, iff $p+q\leq 1$, or
equivalently, iff $pq=qp=0$. Also, two elements $p,q$ in an OML are \emph
{{\rm(}Mackey{\rm)} compatible} \cite{Mac} iff $p=(p\wedge q)\vee(p\wedge q
\sp{\perp})$ (or equivalently, iff $q=(p\wedge q)\vee(p\sp{\perp}\wedge q)$).
It turns out that two projections $p$ and $q$ are compatible in $P$ iff they
are compatible in $E$ iff $pCq$. It is well-known that an OML is a Boolean
algebra iff its elements are pairwise compatible.

\begin{lemma} \label{lm:pCq}
Let $p,q\in P$ with $pCq$. Then{\rm: (i)} $pq=qp\in P$. {\rm(ii)} $pq\leq p,q$.
{\rm (iii)} $p\leq q\Leftrightarrow p=pq$. {\rm (iv)} $p\wedge q=pq$.
\end{lemma}

\begin{proof}
(i) We have $p=p\sp{2}$, $q=q\sp{2}$ and $pCq$ whence $(pq)\sp{2}=pqpq
=p\sp{2}q\sp{2}=pq$, and therefore $pq=qp\in P$. (ii) Since $0\leq p,
1-q$ and $pC(1-q)$ it follows that $0\leq p(1-q)=p-pq$, so $pq\leq p$,
and similarly $pq\leq q$. (iii) Suppose that $p\leq q$. Then as $(q-p)C(1-q)$,
and $0\leq q-p, 1-q$, it follows that $0\leq(q-p)(1-q)=pq-p$, so $p\leq pq$;
hence $p=pq$ by (ii). Conversely, by (ii), if $p=pq$, then $p\leq q$, proving
(iii). (iv) By (ii), $pq$ is a lower bound in $P$ for $p$ and $q$. Suppose
$r\in P$ and $r\leq p,q$. Then $r=rp=rq$ by (iii), so $r(pq)=(rp)q=rq=r$,
whence $r\leq pq$, proving (iv).
\end{proof}

An element in $A$ is called \emph{simple} iff it is a finite linear
combination of pairwise commuting projections. Every simple element in $A$
can be written as a finite linear combination of pairwise orthogonal
projections. It turns out that each element $a\in A$ is the norm limit and
also the supremum of an ascending sequence of pairwise commuting simple
elements \cite[Corollary 8.6 and Theorem 8.9]{Fsa}.

Using SA6, one can prove that if $a\in A$, there exists a unique projection,
denoted by $a\dg\in P$ and called the \emph {carrier} of $a$, such that,
for all $b\in A$, $ab=0\Leftrightarrow a\dg b=0$. (Some authors refer to
$a\dg$ as the \emph{support} of $a$.) It turns out that $ab=0\Leftrightarrow
a\dg b\dg=0\Leftrightarrow b\dg a\dg=0\Leftrightarrow ba=0$. By \cite
[Theorem 2.10]{Fsa}, $a\dg\in CC(a)$, $|a|\dg=a\dg$, and $(a^n)\dg=
a\dg$ for all $n\in\Nat$.

An element $a\in A$ is \emph{invertible} iff there is a (necessarily
unique) element $a\sp{-1}\in A$ such that $aa\sp{-1}=a\sp{-1}a=1$. Using
SA7, one can prove  that $a\in A$ is invertible iff there exists $0<
\epsilon\in\reals$ such that $\epsilon\leq|a|$ \cite[Theorem 7.2]{Fsa}.

In the presence of the remaining axioms, SA8 is equivalent to the
condition that $C(a)$ is norm closed for all $a\in A$ \cite[Theorem 8.11]
{Fsa}.

An element $s\in A$ such that $s\sp{2}=1$ is called a \emph{symmetry}
\cite{FPSym}. Symmetries are in one-to-one correspondence with projections
as follows: If $s$ is a symmetry, then $p=\frac12(s+1)$ is a projection,
and if $p$ is a projection, then $s=2p-1$ is a symmetry. If $s\in A$ is a
symmetry, then $|s|=\|s\|=1$.

A subset $B$ of $A$ is a \emph{sub-synaptic algebra} iff it is a linear
subspace of $A$, $1\in B$, and $B$ is closed under the formation of squares,
square roots, carriers, and inverses. If $B$ is a sub-synaptic algebra of $A$,
then $B$ is a synaptic algebra in its own right under the restrictions to $B$
of the partial order and operations on $A$.  For instance, the commutant $C(B)$
of any subset $B$ of $A$ is a sub-synaptic algebra of $A$.

If $B$ is a commutative subset of $A$, then so is the double commutant $CC(B)$;
moreover, $B\subseteq CC(B)$ and $CC(B)$ is a commutative sub-synaptic algebra of
$A$. In particular, if $a\in A$, then the singleton set $\{a\}$ is commutative so
$a$ belongs to the commutative sub-synaptic algebra $CC(a)$ of $A$. Moreover, if
$b\in CC(a)$, then $b\in CC(b)\subseteq CC(a)$.

By \cite[Section 8]{Fsa}, each element $a$ in a synaptic algebra $A$ both determines
and is determined by a corresponding \emph{spectral resolution}, that is, the ascending
family $(p_{a,\lambda})_{\lambda\in\reals}$ of projections in $P\cap CC(a)$, given by
$p_{a,\lambda}:=1-((a-\lambda)^+)\dg$ for all $\lambda\in\reals$. Moreover, $a=\int_
{L_a-0}^{U_a}\lambda dp_{a,\lambda}$, where the Riemann-Stiltjes type integral converges
in norm, $-\infty<L_a:=\sup\{\lambda\in\reals:p_{a,\lambda=0}\}$, and $U_a:=\inf
\{\lambda\in\reals:p_{a,\lambda}=1\}<\infty$.

The \emph{resolvent set} of $a\in A$ is defined as the set of all real numbers $\mu$ such
that there exists an open interval $I\ni\mu$ in $\mathbb R$ such that $p_{a,\lambda}=p_{a,\mu}$
for all $\lambda \in I$. The \emph{spectrum} of $A$, in symbols $spec(a)$, is defined to be the
complement in $\mathbb R$ of the resolvent set. By \cite[Lemma 3.1]{FPproj}, $\lambda \in
spec(a)$ iff the element $a-\lambda$ is not invertible in $A$. By \cite[Theorem 4.3]{FPsr},
$spec(a)$ is a closed nonempty subset of the closed interval $[L_a,U_a]\subseteq {\mathbb R}$,
$U_a=\sup(spec(a))\in spec(a)$, $L_a=\inf(spec(a))\in spec(a)$. Thus, $spec(a)$ is a nonempty
compact subset of $\mathbb R$. Moreover, $\|a\|=\sup\{|\alpha |:\alpha\in spec(a)\}$.

Also, by \cite[Theorem 8.10]{Fsa}, if $b\in A$ then $bCa$ iff $bCp_{a,\lambda}$  for all
$\lambda\in\reals$, so that two elements in $A$ commute iff the projections in their
respective spectral resolutions commute. If $B$ is a sub-synaptic algebra of $A$, then
the spectral resolution, hence also the spectrum, of each $b\in B$ is the same whether
calculated in $A$ or in $B$.

\begin{definition} \label{df:CVProps}
The synaptic algebra $A$ has the \emph{commutative Vigier} (CV) property iff,
for every bounded ascending sequence $(a\sb{n})\sb{n=1}\sp{\infty}$ of pairwise
commuting elements in $A$, there exists $a\in CC(\{a\sb{n}:n\in\Nat\})$ with
$a\sb{n}\nearrow a$.
\end{definition}

As a consequence of \cite[Lemma 6.6]{FPgh} condition CV is stronger than
axiom SA8. We note that, if $A$ is commutative, then $A$ satisfies condition
CV iff $A$ is monotone $\sigma$-complete, in which case $A$ is a Banach
algebra (Theorem \ref{th:Handelman}).

\begin{remarks} \label{rm:GH}
In view of the discussion in \cite[\S 6]{Fsa}, \emph{in what follows, we can
and shall regard a GH-algebra as a synaptic algebra in which axiom SA8 is
replaced by the stronger CV condition.} By \cite[Lemma 5.4]{FPgh}, if $A$ is
a GH-algebra, then $P$ is a $\sigma$-complete OML.
\end{remarks}

We note that ${\mathcal B}\sp{sa}({\mathfrak H})$ is a GH-algebra. Every
synaptic algebra of finite rank (meaning that there exists $n\in\Nat$ such
that there are $n$, but not $n+1$ pairwise orthogonal nonzero projections
in $P$) is a GH-algebra. According to \cite{FPspin}, a synaptic algebra of
rank 2 is the same thing as a spin factor of dimension greater than $1$.
Thus, GH-algebras of finite rank need not be finite dimensional. Additional
examples of GH-algebras can be found in \cite{Fsa, FPreg}.

Condition (4) in the following definition is an enhancement of
\cite[Definition 2.9]{FPproj} but this does not affect the definition
of a synaptic isomorphism.

\begin{definition} \label{df:morphisms}
Let $A\sb{1}$ and $A\sb{2}$ be
synaptic algebras. A linear mapping $\phi\colon A\sb{1}\to A\sb{2}$
is a \emph{synaptic morphism} iff, for all $a,b\in A\sb{1}$:
\[
\begin{array}{ll}
 {\rm(1)}\  \phi(1)=1. & {\rm(2)}\ \phi(a\sp{2})=\phi(a)\sp{2}. \\
 {\rm(3)}\ aCb\Rightarrow\phi(a)C\phi(b). & {\rm(4)}\ \phi(a\dg)=
 \phi(a)\dg.
\end{array}
\]

\noindent A synaptic morphism $\phi$ is a \emph{synaptic isomorphism}
iff it is a bijection and $\phi\sp{-1}$ is also a synaptic morphism.
If $A\sb{1}$ and $A\sb{2}$ are GH-algebras, then a \emph{GH-algebra
morphism} $\phi\colon A\sb{1}\to A\sb{2}$ is a synaptic morphism such
that, for every ascending sequence of pairwise commuting elements
$(a\sb{n})\sb{n=1}\sp{\infty}$ in $A$,
\[
{\rm(5)}\ a\sb{n}\nearrow a\in A\Rightarrow \phi(a\sb{n})\nearrow
 \phi(a).
\]
A GH-algebra morphism is a \emph{GH-algebra isomorphism} iff it is
a bijection and its inverse is also a GH-algebra morphism.
\end{definition}

\begin{remark} \label{rm:GHisomorphism}
Since a synaptic isomorphism is necessarily an order isomorphism,
if $A\sb{1}$ is a GH-algebra, $A\sb{2}$ is a synaptic algebra, and
there is a synaptic isomorphism $\phi$ from $A\sb{1}$ onto $A\sb{2}$,
then $A\sb{2}$ is a GH-algebra and $\phi$ is a GH-isomorphism.
\end{remark}

\section{Commutative synaptic and GH-algebras}

The commutative sub-synaptic algebra $C(A)$ is called the \emph{center}
of $A$, and $A$ itself is commutative iff $A=C(A)$. By \cite[Theorem 5.12]
{FJPvect}, we have the following characterization of commutative synaptic
algebras.

\begin{theorem}\label{th:commut}
The following conditions are mutually equivalent:
{\rm(i)} The synaptic algebra $A$ is commutative. {\rm(ii)} $A$ is lattice
ordered, hence an order unit normed vector lattice. {\rm(iii)} $E$ is an
MV-effect algebra. {\rm (iv)} $P$ is a Boolean algebra. Moreover, every
Boolean algebra can be realized as the Boolean algebra of projections in
a commutative synaptic algebra.
\end{theorem}

\begin{lemma} \label{lm:simple,etc}
Let $A$ be commutative. Then{\rm: (i)} The vector-lattice notions of absolute
value and positive part coincide with the corresponding notions as defined for
a synaptic algebra. {\rm(ii)} For $a,b\in A$, $a\wedge b=\frac12(a+b-|a-b|)$
and $a\vee b=\frac12(a+b+|a-b|)$. {\rm(iii)} $P$ is precisely the set of
characteristic elements in $A$. {\rm(iv)} The simple elements $a\in A$ in the
sense of an order unit normed vector lattice are precisely the simple elements
in $A$ as defined for a synaptic algebra.
\end{lemma}

\begin{proof}
Part (i) follows from \cite[Remarks 5.5]{FJPvect} and (ii) is a consequence of
\cite[\S 4 and Corollary 5.13]{FJPvect}. To prove (iii), we note that if $p$ is
a characteristic element in $A$, then $p\wedge\sb{A}(1-p)=0$, whence $p\wedge\sb{E}
(1-p)=0$, so $p\in P$ by Lemma \ref{lm:Projections}. Conversely, suppose that
$p\in P$. Then by (ii), $p\wedge\sb{A}(1-p)=\frac12(p+1-p-|p-(1-p)|)=\frac12
(1-|2p-1|)$. But, $2p-1$ is a symmetry, whence $|2p-1|=1$, so $p\wedge\sb{A}
(1-p)=0$, proving (iii). Part (iv) follows from (iii).
\end{proof}

From the results in Section \ref{sc:SAGHA}, Theorem \ref{th:commut}, and Lemma
\ref{lm:simple,etc} we have the following: If $A$ is a commutative synaptic
algebra, then $A$ is a commutative, associative, partially ordered archimedean
real linear algebra with a unity element $1$; it is an order unit normed vector
lattice with order unit $1$; it is a normed linear algebra under the order-unit
norm; it can be considered as its own enveloping algebra; the characteristic
elements in $A$ coincide with the projections in the Boolean algebra $P$; and
the set of simple elements in $A$ is a norm dense commutative sub-synaptic algebra
of $A$.

If $A$ is commutative, then the CV condition reduces to monotone
$\sigma$-completeness, and in view of Remarks \ref{rm:GH}, we have the
following.

\begin{theorem} \label{th:CGH}
A commutative GH-algebra is the same thing as a commutative associative
algebra over $\reals$ that satisfies SA1--SA7 and is monotone $\sigma$-complete.
In particular, a commutative GH-algebra is a Banach algebra.
\end{theorem}

Let $X$ be a compact Hausdorff space. As observed in \cite{FPproj}, $C(X,\reals)$
satisfies all of the synaptic algebra axioms, with the possible exception of the
existence of carriers (i.e., SA6). In fact, by \cite[Theorem 6.3]{FJPstat},
$C(X,\reals)$ is a synaptic algebra iff $X$ is basically disconnected and the
latter condition is equivalent to the condition that $C(X,\reals)$ is monotone
$\sigma$-complete. As we observed above, if a synaptic algebra is commutative,
then it satisfies the CV property (i.e., it is a GH-algebra) iff it is monotone
$\sigma$-complete. Therefore, $C(X,\reals)$ is a synaptic algebra iff
it is a (commutative) GH-algebra.

For a commutative GH-algebra we have the following functional representation
theorem (\cite[Theorem 5.9]{FPproj}, \cite[Theorem 6.5]{FJPstat}). (See \cite
[Theorem 4.1]{FPproj} for a more general functional representation of a
commutative synaptic algebra.)

\begin{theorem}\label{th:ghcom}
Suppose that $A$ is a commutative GH-algebra and let $X$ be the basically
disconnected Stone space of the $\sigma$-complete Boolean algebra $P$. Then
$C(X,\reals)$ is a commutative GH-algebra and there exists a GH-isomorphism
$\Psi:A\to C(X,\reals)$ such that the restriction $\psi$ of $\Psi$ to $P$ is
a Boolean isomorphism of $P$ onto $P(X,{\mathbb R})$ as per Stone's theorem.
\end{theorem}

\begin{theorem}\label{th:veclatGH}
{\rm(i)} Every commutative GH-algebra is a monotone $\sigma$-complete order
unit normed vector lattice. {\rm(ii)} Conversely, if $(V,u)$ is a monotone
$\sigma$-complete order unit normed vector lattice and $B$ is the Boolean
$\sigma$-algebra of characteristic elements in $V$, then a multiplication
$(v,w)\mapsto vw$ can be defined on $V$ in such a way that $V$ becomes a
commutative Banach GH-algebra with order unit $u$ and $B$ is precisely the
set $P$ of projections in $V$.
\end{theorem}

\begin{proof}
Part (i) follows from Theorems \ref{th:commut} and \ref{th:CGH}. To prove (ii),
let $(V,u)$ be a monotone $\sigma$-complete order unit normed vector lattice. By
Theorem \ref{th:Handelman}, $V$ is a Banach space. By Theorem \ref{th:cristTh2},
the set $B$ of characteristic elements of $V$ forms a Boolean $\sigma$-algebra
with a basically disconnected Stone space $X$; moreover, there is isomorphism of
order unit normed vector lattices, $\Psi\colon V\to C(X,\reals)$, of $V$ onto
$C(X,\reals)$ such that the restriction $\psi$ of $\Psi$ to $B$ is a Boolean
isomorphism of $B$ onto $P(X,\reals)$. In particular, $\Psi(u)=\psi(u)=1$. By
\cite[Theorem 6.3]{FJPstat}, $C(X,\reals)$ is a commutative GH-algebra. For
$v,w\in V$, we define the product $vw\in V$ by $vw:=\Psi\sp{-1}\left(\Psi(v)
\Psi(w)\right)$, whereupon $V$ is organized into a commutative, associative,
partially ordered archimedean real linear algebra with a unity element $u$ that
is also an order unit. Obviously, $\Psi\colon V\to C(X,\reals)$ is an isomorphism
of real linear algebras.  Since $C(X,\reals)$ is a commutative GH-algebra, so is
$V$. Also, if $p\in V$, then $p=p\sp{2}$ iff $\Psi(p)=(\Psi(p))\sp{2}$ iff $\Psi(p)
\in P(X,\reals)$ iff $p\in B$.
\end{proof}

\section{States} \label{sc:states}

Just as is the case for any order unit normed space, a \emph{state} on the
synaptic algebra $A$ is defined to be a linear functional $\rho:A\to\reals$ that
is positive ($a\in A\sp{+}\Rightarrow\rho(a)\in\reals\sp{+}$) and normalized
($\rho(1)=1$) \cite[Definition 4.5]{FJPstat}. The set of states on $A$ (the
\emph{state space} of $A$), which is a convex set, is denoted by $S(A)$, and the
set of extreme points in $S(A)$ (\emph{extremal states} on $A$) is denoted by
$Ext(S(A))$.  A state $\rho\in S(A)$ is said to be \emph{$\sigma$-additive} iff,
for every ascending sequence $(a\sb{n})\sb{n\in\Nat}$ in $A$, $a\sb{n}\nearrow
a\in A\Rightarrow \rho(a\sb{n})\nearrow\rho(a)$ in $\reals$. By \cite[Theorems
4.6, 4.9]{FJPstat} $S(A)$ determines both the partial order and the order-unit
norm on $A$.

The next three theorems provide characterizations of extremal states on a
commutative synaptic algebra and on a commutative GH-algebra.

\begin{theorem}{\rm \cite[Theorem 7.1]{FJPstat}} \label{th:extsyn}
Let $A$ be a commutative synaptic algebra, hence a vector lattice. Then for
$\rho\in S(A)$, the following conditions are mutually equivalent{\rm: (i)}
$\rho\in Ext(S(A))$. {\rm(ii)} $\rho:A\to\reals$ is a lattice homomorphism.
{\rm(iii)} $\rho(a\wedge b)=\min\{\rho(a),\rho(b)\}$ for all $a,b\in A^+$.
\end{theorem}

\begin{theorem}{\rm \cite[Theorem 7.2]{FJPstat}} \label{th:extgh1}
Suppose that $A$ is a commutative GH-algebra, $X$ is the compact Hausdorff
basically disconnected Stone space of the $\sigma$-complete Boolean algebra
$P$, $\Psi:A\to C(X,\reals)$ is the synaptic isomorphism of Theorem
\ref{th:ghcom}, and $\rho\in S(A)$. Then the following conditions are mutually
equivalent{\rm:}
\begin{enumerate}
\item $\rho\in Ext(S(A))$.
\item There exists $x\in X$ such that $\rho(a)=(\Psi(a))(x)$ for all $a\in A$.
\item $\rho$ is multiplicative, i.e., $\rho(ab)=\rho(a)\rho(b)$ for all $a,b\in A$.
\item $\rho(p)\in \{ 0,1\}$ for all $p\in P$.
\end{enumerate}
\end{theorem}

\begin{theorem}  \label{th:extgh2}
Let $A$ be a commutative GH-algebra. Then, with the notation of Theorem
\ref{th:extgh1}{\rm: (i)} There is an affine bijection $\rho\leftrightarrow
\gamma$ between states $\rho\in S(A)$ and states $\gamma\in S(C(X,\reals))$
such that $\gamma(f)=\rho(\Psi\sp{-1}(f))$ for all $f\in C(X,\reals)$ and
$\rho(a)=\gamma(\Psi(a))$ for all $a\in A$. In particular, if $\rho
\leftrightarrow\gamma$, then $\rho$ is extremal iff $\gamma$ is extremal.
{\rm(ii)} There is a bijective correspondence $x\leftrightarrow\rho\sb{x}$
between points $x\in X$ and extremal states $\rho\sb{x}$ on $A$ such that
$\rho\sb{x}(a)=(\Psi(a))(x)$ for all $x\in X$ and all $a\in A$.
\end{theorem}

\begin{proof}
Part (i) follows from the fact that $\Psi:A\to C(X,\reals)$ is a synaptic
isomorphism.

(ii)  According to \cite[Theorem 4.10 (iii), (iv)]{FJPstat}, each $x\in X$
induces a state $\gamma_x\in S(C(X,\reals))$ such that $\gamma_x(f)=f(x)$
for all $f\in C(X,\reals)$; moreover, $Ext(S(C(X,\reals))=\{\gamma_x:x\in X\}$.
For each $x\in X$, let $\rho\sb{x}$ be the extremal state on $A$ corresponding
to the extremal state $\gamma\sb{x}$ on $C(X,\reals)$ according to (i), so that
$\rho\sb{x}(a)=\gamma\sb{x}(\Psi(a))=(\Psi(a))(x)$ for all $a\in A$. If $x,y
\in X$ and $\rho\sb{x}=\rho\sb{y}$, then $f(x)=f(y)$ for all $f\in C(X,\reals)$,
and since $C(X,\reals)$ separates points in $X$, it follows that $x=y$.
\end{proof}

\begin{definition} \label{df:hatetc}
Let $A$ be a commutative GH-algebra. Then, with the notation of Theorem
\ref{th:extgh2}, we define ${\hat a}:=\Psi(a)\in C(X,\reals)$ for all $a
\in A$. Also, in view of Theorem \ref{th:extgh2} (ii), we can and do
identify each point $x\in X$ with the extremal state $\rho\sb{x}$ on $A$,
so that $x(a)={\hat a}(x)$ for all $a\in A$ and all $x\in X$.
\end{definition}

\section{A Loomis-Sikorski type theorem for commutative GH-algebras}

The classical Loomis-Sikorski theorem for Boolean $\sigma$-algebras \cite{Loo,
Sik} has been extended to $\sigma$MV-algebras in \cite{BW, Dvls, Munls} (Theorem
\ref{th:LSmv} below).  In \cite{Dvls}, it was also generalized to Dedekind
$\sigma$-complete abelian $\ell$-groups (see also \cite[\S 7.14]{DvPu}). In
this section, we shall generalize the Loomis-Sikorski theorem to commutative
GH-algebras (Theorem \ref{th:LSCGH} below).

We shall be dealing with a nonempty set $X$ and with the set $\reals\sp{X}$
of all functions $f\colon X\to\reals$. The partial order $\leq$ is defined
pointwise on $\reals\sp{X}$, i.e., for $f,g\in\reals\sp{X}$, $f\leq g$ iff
$f(x)\leq g(x)$ for all $x\in X$. Similarly, the operations $f\pm g$,
$fg$, $\alpha f$ for $\alpha\in\reals$, $\min(f,g)$ and $\max(f,g)$ are
understood to be defined pointwise on $\reals\sp{X}$. The constant functions
$x\mapsto 0$ and $x\mapsto 1$ will be denoted by $0,1\in\reals\sp{X}$. Unless
$X$ is finite, $1$ cannot be an order unit in $\reals\sp{X}$; however, $\reals
\sp{X}$ is an archimedean vector lattice with $f\vee g=\max(f,g)$ and $f\wedge
g=\min(f,g)$, and it is also a commutative linear algebra with unit $1$ under
the pointwise operations. The set $X$ may be required to satisfy certain
special conditions, e.g., it might be stipulated that $X$ is a Stone space.

Let $(f\sb{n})\sb{n=1}\sp{\infty}$ be a sequence in $\reals\sp{X}$ that is
bounded above in $\reals\sp{X}$ (i.e., there exists $g\in\reals\sp{X}$ such
that $f\sb{n}\leq g$ for all $n\in\Nat$). Then by Dedekind completeness of
$\reals$, the pointwise supremum of $(f\sb{n})\sb{n=1}\sp{\infty}$ exists and
will be denoted by $\sup\sb{n}f\sb{n}\in\reals\sp{X}$. Thus, $(\sup\sb{n}f
\sb{n})(x):=\sup\sb{n}(f\sb{n}(x))$ for each $x\in X$, and it is clear that
$\sup\sb{n}f\sb{n}=\bigvee\sb{n\in\Nat}f\sb{n}$ in the poset $\reals\sp{X}$.
Thus, $\reals\sp{X}$ is Dedekind $\sigma$-complete.

We shall be considering nonempty subsets ${\mathcal T}\subseteq\reals\sp
{X}$ that satisfy certain prescribed conditions, e.g., regarding their poset
structure, being closed under various operations on $\reals\sp{X}$,
consisting only of specified kinds of functions, etc.

\begin{definition} \label{df:tribe}
A \emph{tribe} on the nonempty set $X$ is a set ${\mathcal T}\subseteq
\reals\sp{X}$ that  satisfies the following conditions: (1) $f\in{\mathcal T}
\Rightarrow 0\leq f\leq 1$. (2) $0\in{\mathcal T}$. (3) $f\in{\mathcal T}
\Rightarrow 1-f\in{\mathcal T}$. (4) $f,g\in{\mathcal T}\Rightarrow\min(f+g,1)
\in{\mathcal T}$. (5) If $(f\sb{n})\sb{n=1}\sp{\infty}$ is an ascending
sequence of functions in ${\mathcal T}$ and $f\sb{n}\nearrow f\in\reals\sp{X}$
pointwise (i.e., $f=\sup\sb{n}f\sb{n}$) then $f\in{\mathcal T}$.
\end{definition}

By \cite[Proposition 7.1.6]{DvPu}, every tribe is a $\sigma$MV-algebra that is
closed under pointwise suprema of sequences of its elements.

If $X$ is any nonempty set, then the set ${\mathcal T}\sb{e}$ of all functions
$f\in\reals\sp{X}$ such that $0\leq f\leq 1$ is a convex tribe and every tribe
${\mathcal T}\subseteq\reals\sp{X}$ satisfies ${\mathcal T}\subseteq{\mathcal T}
\sb{e}$. Moreover, the intersection of any family of tribes (convex tribes)
${\mathcal T}\subseteq\reals\sp{X}$ is again a tribe (a convex tribe). Thus, if
${\mathcal T}\sb{0}\subseteq{\mathcal T}\sb{e}$, then the intersection of all
tribes (convex tribes) ${\mathcal T}$ with ${\mathcal T}\sb{0}\subseteq{\mathcal T}$
is a tribe (a convex tribe) called the tribe (the convex tribe) \emph{generated by}
${\mathcal T}\sb{0}$.

For purposes of this paper, by a \emph{functional representation} of a partially
ordered algebraic structure ${\mathcal P}$ we shall mean a subset ${\mathcal T}
\subseteq\reals\sp{X}$ with partially ordered and algebraic structure corresponding
to that of ${\mathcal P}$ together with a surjective morphism $h\colon{\mathcal T}
\to{\mathcal P}$ of ${\mathcal T}$ onto ${\mathcal P}$.  For instance,
the Loomis-Sikorski functional representation theorem for $\sigma$MV-algebras
\cite[Theorem 1.7.22]{DvPu} is as follows.

\begin{theorem}\label{th:LSmv}
For each $\sigma$MV-algebra $M$ there exists a nonempty set $X$, a tribe
${\mathcal T}\subseteq\reals\sp{X}$, and a surjective $\sigma$MV-homomorphism
$h\colon{\mathcal T}\to M$ from ${\mathcal T}$ onto $M$.
\end{theorem}

We note that Theorem \ref{th:ghcom} yields a functional representation for a
commutative GH-algebra $A$ in terms of a basically disconnected Stone space $X$
and the commutative GH-algebra $C(X,\reals)\subseteq\reals\sp{X}$; indeed it
provides a GH-isomorphism $\Psi\sp{-1}\colon C(X,\reals)\to A$. However, this
representation may suffer from the defect that there may be bounded sequences
$(f\sb{n})\sb{n\in\Nat}$ in $C(X,\reals)$ such that the supremum $\bigvee\sb{n\in\Nat}
f\sb{n}$ in $C(X,\reals)$ exists but is not the pointwise supremum $\sup\sb{n}f\sb{n}$.
In the Loomis-Sikorski functional representation of $A$ (Theorem \ref{th:LSCGH} below),
this defect is ameliorated by employing the following definition (which is an extension
of the notion of a g-tribe \cite[p. 465]{DvPu}).

\begin{definition}\label{df:ghtribe}
A \emph{gh-tribe} on the nonempty set $X$ is a set ${\mathcal T}\subseteq
\reals\sp{X}$ such that:
\begin{enumerate}
\item[(1)] If $f\in{\mathcal T}$, then $f$ is bounded, i.e., there exist $\alpha,
 \beta\in\reals$ with $\alpha\leq f(x)\leq\beta$ for all $x\in X$.
\item[(2)] $0,1\in{\mathcal T}$.
\item[(3)] $f+g\in{\mathcal T}$ whenever $f,g\in {\mathcal T}$.
\item[(4)] For every $\alpha\in\reals$ and every $f\in{\mathcal T}$, $\alpha f
 \in{\mathcal T}$.
\item[(5)] If $(f_n)_{n=1}^\infty$ is a sequence of functions in ${\mathcal T}$
 for which there exists $f\in{\mathcal T}$ with $f_n\leq f$ for all $n\in\Nat$,
 then the pointwise supremum $\sup_n f_n$ belongs to ${\mathcal T}$.
\end{enumerate}
\end{definition}

Let ${\mathcal T}$ be a gh-tribe on $X$. In (5), it is clear that the pointwise
supremum $\sup_n f_n$ is also the supremum $\bigvee_{n\in\Nat}f_n$ in
${\mathcal T}$, so ${\mathcal T}$ is Dedekind $\sigma$-complete. Suppose $f,g
\in{\mathcal T}$. By (1), there exists $\beta\in\reals$ with $f(x),g(x)\leq\beta$
for all $x\in X$, so $f,g\leq \beta1$. Thus the pointwise supremum $\max(f,g)$ of
$f$ and $g$ belongs to ${\mathcal T}$ and it is the supremum $f\vee g$ of $f$ and
$g$ in ${\mathcal T}$. As $f\mapsto -f$ is an involution on ${\mathcal T}$, it
follows that the infimum $f\wedge g$ exists in ${\mathcal T}$, and $f\wedge g=
\min(f,g)$, whence ${\mathcal T}$ is a vector lattice. By (1), $1$ is an order
unit in ${\mathcal T}$, and since $R\sp{X}$ is archimedean, so is ${\mathcal T}$.
Therefore, in view of Theorem \ref{th:OUNVL}, every gh-tribe is a Dedekind
$\sigma$-complete Banach order unit normed vector lattice with order unit $1$.

Let $X$ be any nonempty set. Then the set ${\mathcal T}\sb{b}$ of all bounded
functions in $\reals\sp{X}$ is a gh-tribe on $X$ and every gh-tribe ${\mathcal T}$
on $X$ satisfies ${\mathcal T}\subseteq{\mathcal T}\sb{b}$. Moreover, the
intersection of any family of gh-tribes on $X$ is again a gh-tribe on $X$. Thus,
if ${\mathcal T}\sb{0}\subseteq{\mathcal T}\sb{b}$, then the intersection of all
gh-tribes ${\mathcal T}$ with ${\mathcal T}\sb{0}\subseteq{\mathcal T}$ is a
gh-tribe called the gh-tribe \emph{generated by} ${\mathcal T}\sb{0}$.

Clearly, every order unit normed vector lattice ${\mathcal T}\subseteq\reals\sp{X}$
with order unit $1$ satisfies ${\mathcal T}\subseteq{\mathcal T}\sb{b}$. Moreover,
the intersection of any family of such order unit normed vector lattices ${\mathcal T}
\subseteq\reals\sp{X}$ is again an order unit normed vector lattice with order unit
$1$. If ${\mathcal T}\sb{0}\subseteq{\mathcal T}\sb{b}$, then the intersection of all
order unit normed vector lattices ${\mathcal T}$ with order unit $1$ such that
${\mathcal T}\sb{0}\subseteq{\mathcal T}$ is called the order unit normed vector
lattice \emph{generated by} ${\mathcal T}\sb{0}$.

Let ${\mathcal T}$ be a gh-tribe on $X$. Since ${\mathcal T}$ is a Dedekind
$\sigma$-complete Banach order unit normed vector lattice, we infer from Theorem
\ref{th:veclatGH} that there exists a multiplication operation $(f,g)\mapsto fg$
on ${\mathcal T}$ such that ${\mathcal T}$ forms a commutative Banach GH-algebra
and the set of characteristic elements in ${\mathcal T}$ coincides with the Boolean
$\sigma$-algebra ${\mathcal P}$ of projections (idempotents) in ${\mathcal T}$.

\begin{lemma} \label{lm:pointwisemult}
Let ${\mathcal T}$ be a gh-tribe on $X$ and let $f,g,p\in{\mathcal T}$.
Then{\rm: (i)} $|f|\in{\mathcal T}$ is the pointwise absolute value.  {\rm(ii)}
Norm limits in ${\mathcal T}$ are pointwise limits. {\rm(iii)} $p\in{\mathcal P}
\Leftrightarrow p(x)\in\{0,1\}$ for all $x\in X$. {\rm(iv)} Multiplication on
${\mathcal T}$ is pointwise.
\end{lemma}

\begin{proof}
(i) $|f|(x)=(f\vee(-f))(x)=\max(f(x),-f(x))=|f(x)|$ for all $x\in X$.

(ii) Suppose that $(f\sb{n})\sb{n\in\Nat}\subseteq{\mathcal T}$. By Lemma
\ref{lm:rhovsnorm} (ii) and (i), $\lim\sb{n\rightarrow\infty}f\sb{n}=f$ iff
$\lim\sb{n\rightarrow\infty}f\sb{n}(x)=f(x)$ for all $x\in X$.

(iii) We have $p\in{\mathcal P}$ iff $p\wedge(1-p)=0$ iff $\min(p(x),1-p(x))=0$
for all $x\in X$ iff (iii) holds.

(iv) Let $p,q\in{\mathcal P}$. Then by Lemma \ref{lm:pCq} and (iii),
\[
(pq)(x)=(p\wedge q)(x)=\min(p(x),q(x))=p(x)q(x)\text{\ for all\ }x\in X.
\]
Therefore, $pq$ is the pointwise product of $p$ and $q$, and it follows that
the product of simple elements in ${\mathcal T}$ is the pointwise product. There
are ascending sequences $(s\sb{n})\sb{n\in\Nat}$ and $(t\sb{n})\sb{n\in\Nat}$
of simple elements in ${\mathcal T}$ such that $\lim\sb{n\rightarrow\infty}s\sb{n}=
f$ and $\lim\sb{n\rightarrow\infty}t\sb{n}=g$, whence $fg=\lim\sb{n\rightarrow\infty}
(s\sb{n}t\sb{n})$, and by (ii),
\[
(fg)(x)=\lim\sb{n\rightarrow\infty}(s\sb{n}t\sb{n})(x)=\lim\sb{n\rightarrow\infty}
\left(s\sb{n}(x)t\sb{n}(x)\right)=\left(\lim\sb{n\rightarrow\infty}
 s\sb{n}(x)\right)\left(\lim\sb{n\rightarrow\infty}t\sb{n}(x)\right)
\]
\[
=f(x)g(x)\text{\ for all\ }x\in X. \qedhere
\]
\end{proof}

For a gh-tribe ${\mathcal T}$ on $X$ we define
\[
{\mathcal B}({\mathcal T}):=\{D\subseteq X:\chi_D\in{\mathcal T}\}
\]
and note that ${\mathcal B}({\mathcal T})$ is a $\sigma$-field of subsets of $X$.
Also, the unit interval ${\mathcal T}[0,1]:=\{f\in{\mathcal T}:0\leq f\leq 1\}
\subseteq\reals\sp{X}$ is a convex $\sigma$MV-algebra, which is also a convex tribe.
According to \cite[Theorem 7.4]{DvPucond}, the following statement, which is an
extension of \cite{BuKl} for tribes, holds (see also \cite{Pspec}).

\begin{theorem}\label{th:bukl} Let ${\mathcal T}$ be a gh-tribe on $X$. Then every
$f\in{\mathcal T}$ is ${\mathcal B}({\mathcal T})$-measurable, and for every
$\sigma$-additive state $\rho$ on ${\mathcal T}$, $\rho(f)=\int\sb{X}f(x)\mu(dx)$, where
$\mu$ is the probability measure on ${\mathcal B}({\mathcal T})$ defined by $\mu(D):
=\rho(\chi_D)$ for every $D\in{\mathcal B}({\mathcal T})$.
\end{theorem}

We can now state and prove our Loomis-Sikorski theorem for commutative GH-algebras.

\begin{theorem} \label{th:LSCGH} {\bf (The GH-Loomis-Sikorski theorem.)}
Let $A$ be a commutative GH-algebra and let $X$ be the basically disconnected
Stone space of the $\sigma$-complete Boolean algebra $P$ of projections in
$A$. Then there exists a gh-tribe ${\mathcal T}$ on $X$ such that
$C(X,\reals)\subseteq{\mathcal T}$ and there exists a surjective morphism
$h$ of GH-algebras from $\mathcal T$ onto $A$.
\end{theorem}

\begin{proof}
Noting that the effect algebra $E\subseteq A$ is a convex $\sigma$MV-algebra,
we begin by proving a version of Theorem \ref{th:LSmv} involving convexity.  As
per Definition \ref{df:hatetc}, we shall identify $X$ with the set $Ext(S(A))$ of
all extremal states on $A$. By Theorem \ref{th:ghcom}, there exists a GH-isomorphism
$\Psi:A\to C(X,\reals)$ that restricts to a Boolean isomorphism $\psi$ of $P$ onto the
Boolean algebra $P(X,\reals)\subseteq C(X,\reals)$. As in Definition \ref{df:hatetc},
for $a\in A$, we write $\hat{a}:=\Psi(a)$, so that $\hat{a}(x)=x(a)$ for all $x\in X$.

Let ${\mathcal T}_1$ be the convex tribe generated by $C(X,{\mathbb R})[0,1]=
\{\hat{a}\in C(X,{\mathbb R}): a\in E\}$. If $f\in{\mathcal T}\sb{1}$, define
$N(f):=\{x\in X:f(x)\neq 0\}$.

Let ${\mathcal T}_1\sp{\prime}$ be the set of functions $f\in{\mathcal T}_1$ with
the property that, for some $a\in E$, $N(f-\hat{a})$ is a meager set (a countable
union of nowhere dense sets).  If $a_1$ and $a_2$ are two elements of $E$ such that
$N(f-\widehat{a_1})$ and $N(f-\widehat{a_2})$ are meager sets, then $N(\widehat{a_1}
-\widehat{a_2})\subseteq N(f-\widehat{a_1})\cup N(f-\widehat{a_2})$ is a meager set.
By the Baire theorem, a non-empty open subset   of a compact Hausdorff space cannot
be a meager set, whence $\widehat{a_1}=\widehat{a_2}$, and it follows that such an
$a\in E$ is unique.  Therefore the mapping $h\sb{1}:{\mathcal T}_1\sp{\prime}\to E$
defined by $h\sb{1}(f)=a$ iff $a\in E$ and $N(f-\hat{a})$ is meager is well defined
and maps ${\mathcal T}_1\sp{\prime}$ onto $E$.

Proceeding as in the proof of \cite[Theorem 7.1.22]{DvPu}, we deduce that
${\mathcal T}_1\sp{\prime}$ is a tribe. Clearly, if $0\leq\alpha \leq 1$ and
$N(f-\hat{a})$ is meager, then $N(\alpha f-\alpha\hat{a})$ is meager, so we
may infer that ${\mathcal T}_1\sp{\prime}$ is a convex tribe that contains
$C(X,\reals)[0,1]$ and that the mapping $h\sb{1}$ preserves the convex structure.
From this it follows that ${\mathcal T}\sb{1}\subseteq{\mathcal T}\sb{1}\sp{\prime}$,
and hence ${\mathcal T}\sb{1}={\mathcal T}\sb{1}\sp{\prime}$. Moreover, the mapping
$h\sb{1}:{\mathcal T}\sb{1}\to E$ corresponds to the surjective $\sigma$MV-algebra
homomorphism of Theorem \ref{th:LSmv}.

Now let $\mathcal T$ be the gh-tribe of functions $f:X\to {\mathbb R}$ generated
by $C(X,\reals)$. We note that, by \cite[Proposition 7.1.25]{DvPu}, ${\mathcal
T}$ is the set of all bounded Baire measurable functions on $X$. We claim that
${\mathcal T}_1={\mathcal T}[0,1]=\{f\in{\mathcal T}:0\leq f\leq 1\}$. From
$C(X,\reals)[0,1]\subseteq C(X,{\mathbb R})$ we deduce that
\begin{equation}\label{eq: 1}
{\mathcal T}_1\subseteq {\mathcal T}[0,1].
\end{equation}
Let ${\mathcal V}$ be the order unit normed vector lattice generated by
${\mathcal T}_1$. By \cite[Theorem 16.9]{Good}, ${\mathcal V}$ is monotone
$\sigma$-complete since ${\mathcal T}_1$ is. It follows that ${\mathcal V}$ is
a gh-tribe, and from (\ref{eq: 1}) we conclude that ${\mathcal V}\subseteq
{\mathcal T}$. Since $C(X,\reals)\subseteq {\mathcal V}$, we conclude that
${\mathcal V}={\mathcal T}$, and hence ${\mathcal T}_1={\mathcal T}[0,1]$.

Using arguments extracted from Mundici's proof \cite{Munint} \cite[Corollary 5.3.8]
{DvPu} of the categorical equivalence of unital abelian $\ell$-groups and MV-algebras,
we conclude that there exists an extension $h\colon{\mathcal T}\to A$ of the mapping
$h\sb{1}$ such that $h$ is an $\ell$-group homomorphism, i.e., (i) $h(f)\pm h(g)
=h(f\pm g)$, (ii) $h\left(\max(f,g)\right)=h(f)\vee h(g)$, and (iii) $h(1)=1$. To
prove that $h(\sup_nf_n)=\sup_n h(f_n)$ whenever $\sup_nf_n\in{\mathcal T}$ we use
the same setup as in the proof of \cite[Theorem 7.1.24]{DvPu}. By \cite[Lemma 4.2]
{FJPstat}, $h$ is a positive linear mapping from ${\mathcal T}$ onto $A$. Since every
extremal state is multiplicative, we have for all $a,b\in E$, $\widehat{ab}=\hat{a}
\hat{b}$, and owing to the inequalities
\begin{eqnarray*}
\mid fg-\widehat{ab}\mid=\mid fg-\hat{a}\hat{b}\mid\leq\mid fg-f\hat{b}\mid+\mid f\hat{b}
 -\hat{a}\hat{b}\mid=\mid f\mid\mid g-\hat{b}\mid+\mid f-\hat{a}\mid\mid\hat{b}\mid,
\end{eqnarray*}
$N(fg-\hat{a}\hat{b})$ is a meager set provided that $N(f-\hat{a})$ and $N(g-\hat{b})$
are meager. From this it follows that $h\sb{1}(fg)=h\sb{1}(f)h\sb{1}(g)$, so that
$h\sb{1}:{\mathcal T}\sb{1}\to E$ preserves multiplication. Considering first  $f,g\in
{\mathcal T}^+$, and then taking into account the decomposition $f=f^+-f^-$ for all
$f\in{\mathcal T}$, it is easy to see that $h$ preserves multiplication.  We conclude
that $h$ is a GH-algebra morphism from the gh-tribe ${\mathcal T}$ onto $A$.
\end{proof}

Notice that the $h$ is \emph{regular} in the sense that $h(f)=0$ iff $h(\chi_{N(f)})=0$.
We define the \emph{regular representation} of the commutative GH-algebra $A$ to be
the triple $(X,{\mathcal T},h)$ in Theorem \ref{th:LSCGH}.

\section{Observables and a continuous functional \newline calculus for a GH-algebra}
\label{sc:ObFC}

\begin{assumptions}
In this section, we assume that $A$ is a GH-algebra and we fix an element $a\in A$.
Thus $a$ belongs to the commutative Banach GH-algebra $CC(a)$. Also $X$ denotes the
basically disconnected Stone space of the Boolean $\sigma$-algebra $P\cap CC(a)$ of
all projections in $CC(a)$, and $\Psi:CC(a)\to C(X,\reals)$ is the GH-isomorphism
of Theorem \ref{th:ghcom}.
\end{assumptions}

\begin{lemma} \label{le:specg}
If $g\in C(X,\reals)$, then $\spec(g)=\{g(x):x\in X\}$.
\end{lemma}

\begin{proof}
If $\lambda\in\reals$, then $g-\lambda$ is invertible in $C(X,\reals)$ iff
there exists $0<\epsilon\in\reals$ such that $|g(x)-\lambda|\geq\epsilon$ for
all $x\in X$, whence $g-\lambda$ is not invertible in $C(X,\reals)$ iff, for
each $n\in\Nat$, there exists $x\sb{n}\in X$ with $|g(x\sb{n})-\lambda|<
\frac{1}{n}$. Since $X$ is compact, there is a subsequence $(x_{n_k})$ of
$(x_n)$ such that $\lim\sb{k\rightarrow\infty}x_{n_k}=x$, for some $x\in X$.
Since $g$ is continuous, we have $\lambda=\lim\sb{k\rightarrow\infty}g(x_{n_k})
=g(x)$ and it follows that $\spec(b)\subseteq\{g(x):x\in X\}$. The opposite
inclusion is obvious.
\end{proof}

For a subset $D\subseteq X$,  $D^-$ denotes the closure and $D^i$ denotes the interior
of $D$. If $g\in C(X,\reals)$, then by \cite[Theorem C, p. 217]{Hal}, $N(g):=
\{x\in X:g(x)\neq 0\}$ is an open $F_{\sigma}$ set, and since $X$ is basically
disconnected, $N(g)^-$ is a clopen set and the carrier of $g$ is $g\dg=
\chi_{N(g)^-}$.

\begin{theorem}\label{th:specstuff}
Let $b\in CC(a)$ with $g:=\Psi(b)\in C(X,\reals)$. Then{\rm:}
\begin{enumerate}
\item $CC(b)\subseteq CC(a)$, so the spectral resolution $(p_{b,\lambda})
 \sb{\lambda\in\reals}$ of $b$ is contained in $CC(a)$.
\item For all $\lambda\in\reals$, $\Psi(p_{b,\lambda})=\chi_{\left(g^{-1}
(-\infty,\lambda]\right)^i}$.
\item $\spec(b)=\spec(g)=\{g(x):x\in X\}$.
\end{enumerate}
\end{theorem}

\begin{proof} (i) That $CC(b)\subseteq CC(a)$ is easily verified.

(ii) We have $p_{b,\lambda}=1-((b-\lambda)^+)\dg$. Also $\Psi((b-\lambda)^+)=
(g-\lambda)^+$, and $\{x\in X:(g-\lambda)^+(x)\not=0\}=\{x\in X:g(x)-\lambda>0\}
=g\sp{-1}(\lambda,\infty)$, whence $((g-\lambda)\sp{+})\dg=\chi\sb{(g\sp{-1}
(\lambda,\infty))\sp{-}}$. Therefore, $\Psi(p_{b,\lambda})=1-\left((g-\lambda)
\sp{+}\right)\dg=\chi_{\left(g^{-1}(-\infty,\lambda]\right)^i}$.

(iii) Let $\lambda\in\reals$. Recall that the spectrum of $b$ in $A$ is the same
as the spectrum of $b$ in $CC(a)$. Thus, $\lambda\in\spec(b)$ iff $b-\lambda$ is not
invertible in $A$ iff $b-\lambda$ is not invertible in $CC(a)$. Since $\Psi$ is a
synaptic isomorphism, $b-\lambda$ is not invertible in $CC(a)$ iff $\Psi(b-\lambda)
=g-\lambda$ is not invertible in $C(X,\reals)$, i.e., iff $\lambda\in\spec(g)$. Thus,
$\spec(b)=\spec(g)$, and (iii) follows from Lemma \ref{le:specg}.
\end{proof}

We recall that the set of projections $P$ in a GH-algebra is a $\sigma$-ortho\-mo\-du\-lar
lattice ($\sigma$-OML) \cite[Theorem 5.4]{FPgh}. On any $\sigma$-OML $L$, an \emph{observable}
$\xi$ is a $\sigma$-morphism from a Boolean algebra ${\mathcal B}$ into $L$, that is,
$\xi:{\mathcal B} \to L$ has the following properties: (i) $\xi(1)=1$, (ii) $\xi(a\cup b)
=\xi(a)\vee \xi(b)$ whenever $a,b\in {\mathcal B}$ are such that $a\cap b=0$, and (iii) if
$a\in {\mathcal B}, (a_n)_{n\in\Nat}\subseteq{\mathcal B}$, and $a_n\nearrow a$, then $\xi(a_i)
\nearrow \xi(a)$. An observable is \emph{real} if ${\mathcal B}={\mathcal B}(\reals)$ is
the Boolean $\sigma$-algebra of Borel subsets of $\reals$. If $\rho$ is a $\sigma$-additive
state on $L$, then $\rho\circ\xi:{\mathcal B}\to [0,1]$ is a probability measure on
${\mathcal B}$, which is called the \emph{distribution of $\xi$} in the state $\rho$. The
\emph{expectation} of a real observable $\xi$ in a state $\rho$ is then given by $\rho(\xi)
:=\int_{\mathbb R}\lambda\rho(\xi(d\lambda))$. (For more details on observables on
$\sigma$-OMLs see e.g. \cite{Var, PtPu}.) In the next theorem, we show how each element $a$
of the GH-algebra $A$ corresponds to a real observable $\xi_a$ on the $\sigma$-OML $P$ of
projections in $A$.

\begin{theorem}\label{th:observ} Let $(X,{\mathcal T},h)$ be the regular Loomis-Sikorski
representation of $CC(a)$ and choose $f\sb{a}\in{\mathcal T}$ with $h(f_a)=a$. Define
$\xi_a\colon{\mathcal B}(\reals)\to P$ by $\xi_a(D)=h(\chi\sb{f\sb{a}\sp{-1}(D)})$ for
every $D\in{\mathcal B}(\reals)$. Then $\xi_a$ is a real observable on $P$ that is
independent of the choice of $f_a$; moreover, $\xi_a$ is the unique real observable on
$P$ such that $\xi_a((-\infty,\lambda])=p\sb{a,\lambda}$ for all $\lambda \in\reals$.
\end{theorem}

\begin{proof} Since $f\sb{a}\in {\mathcal T}$ is ${\mathcal B}({\mathcal T})$-measurable
(Theorem \ref{th:bukl}), $f_a^{-1}$ maps ${\mathcal B}(\reals)$ to ${\mathcal B}
({\mathcal T})$, and since $h$ maps $\{\chi_J:J\in{\mathcal B}({\mathcal T})\}$ onto
$P\cap CC(a)\subseteq P$, the mapping $\xi_a\colon{\mathcal B}(\reals)\to P$ is
well-defined. It is straightforward to check that $\xi_a$ is a real observable

Let $f,g\in {\mathcal T}$ with $h(f)=h(g)$. Then, if $\Delta$ is the set-theoretic
symmetric difference on $X$, we have $f^{-1}(D)\Delta g^{-1}(D)\subseteq N(f-g)$,
where $N(f-g)$ is a meager set, whence the characteristic function of $f^{-1}(D)
\Delta g^{-1}(D)$ belongs to the kernel of $h$. Therefore, $h(\chi\sb{f\sp{-1}(D)})
=h(\chi\sb{g\sp{-1}(D)})$ for all $D\in{\mathcal B}(\reals)$. In particular, if
$a=h(f\sb{a})=h(g)$, then $h(\chi\sb{f_a\sp{-1}(D)})=h(\chi\sb{g\sp{-1}(D)})$, whence
$\xi_a$ is independent of the choice of $f\sb{a}$.

By Theorem \ref{th:specstuff} (ii) with $b=a$ and $\Psi(a)=g\in C(X,\reals)$, we have
$p\sb{a,\lambda}=h(\chi\sb{(g\sp{-1}(-\infty,\lambda])\sp{i}})$. But, since the set
difference between a set and its interior is  meager, we have $p_{a,\lambda}=
h(\chi\sb{g\sp{-1}(-\infty,\lambda]})=\xi_a((-\infty,\lambda])$ for all $\lambda
\in\reals$. But, for any real observable $\xi$ on $P$, it is clear that the values
of $\xi$ on intervals $(-\infty,\lambda]$ for $\lambda\in\reals$, uniquely determine
the observable $\xi$.
\end{proof}

We note that since every element in $A$ is uniquely determined by its spectral
resolution, the observable $\xi_a$ uniquely determines $a$.

\begin{corollary}\label{co:stat} For every $\sigma$-additive state $\rho$ on $A$,
\[
\rho(a)=\int_{\mathbb R} \lambda \rho(\xi_a(d\lambda)), a\in A,
\]
that is, $\rho(a)$ is equal to the expectation of the observable $\xi_a$ in the state $\rho$.
\end{corollary}

\begin{proof}
Let $\rho$ be a $\sigma$-additive state on $A$. The restriction of $\rho$ to $CC(a)$ is a
$\sigma$-additive state on $CC(A)$. Let $(X,{\mathcal T},h)$ be the regular representation of
$CC(a)$. Define $\hat{\rho}:{\mathcal T}\to\reals$ by $\hat{\rho}(f):=\rho(h(f))$. If $h(f)=0$,
then $\hat{\rho}(f)=0$, hence $\ker(\hat{\rho})\supseteq \ker(\rho)$. Since $\rho$ is a
$\sigma$-additive state and $h$ is a morphism of GH-algebras, $\hat{\rho}$ is $\sigma$-additive
state on ${\mathcal T}$ vanishing on the kernel of $h$.

Choose $f\sb{a}\in{\mathcal T}$ with $h(f_a)=a$, so that $\rho(a)=\hat{\rho}(f_a)$. From the
integral representation of $\hat{\rho}$ (Theorem \ref{th:bukl}), and the integral transformation
theorem, we obtain for every $f\in {\mathcal T}$, $\hat{\rho}(f)=\int_X f(x)\hat{\rho}(dx)=
\int_{\mathbb R}\lambda\hat{\rho}(f^{-1}(d\lambda))$. Thus, $\rho(a)=\int_{\mathbb R}
\lambda\rho(h(f_a^{-1}(d\lambda))=\int_{\mathbb R}\lambda\rho(\xi_a(d\lambda))$.
\end{proof}

Our continuous functional calculus for the GH-algebra $A$ is based on the following
definition.

\begin{definition} \label{df:f(a)}
Let $f\in C(\spec(a),\reals)$ and let $g:=\Psi(a)\in C(X,\reals)$. By Theorem
\ref{th:specstuff} (iii), $g(x)\in\spec(a)$ for all $x\in X$, whence the composition
$f\circ g$ is defined and $f\circ g\in C(X,\reals)$. Thus we can and do define
the element $f(a)\in CC(a)$ by $f(a):=\Psi\sp{-1}(f\circ g)$.
\end{definition}

Notice that if $q(t)=\alpha_0+\alpha_1 t+\alpha_2t^2+\ldots \alpha_nt^n$ is a real polynomial, then
the element $q(a)=\alpha_0+\alpha_1a+\alpha_2a^2+\ldots+ \alpha_na^n$ coincides with the  element  $q(a)$ defined
by Definition \ref{df:f(a)}.  Indeed, if $g:=\Psi(a)\in C(X,\reals)$, then
$q\circ g(x)=q(g(x))=\alpha_0+\alpha_1 g(x)+\alpha_2g(x)^2+\ldots +\alpha_ng(x)^n$,
and hence $\Psi(q(a))=q\circ g$.

Notice that by the Stone-Weierstrass theorem, to every
 $f\in C(spec(a),\reals)$
there is a sequence of real polynomials $q_n$ such that $q_n\rightarrow f$ uniformly on $spec(a)$.
It is easy to check that $f(a)$ is the norm limit of the polynomials $q_n(a)$ in $A$.

\begin{theorem}\label{th:RSint}  Let $f\in C(spec(a),\reals)$.
Then $f(a)=\int_{L_a-0}^{U_a}f(\lambda)dp_{a,\lambda}$.
\end{theorem}

\begin{proof} We have $a=\int_{L_a-0}^{U_a}\lambda dp_{a,\lambda}$. The integral on the
right is a norm-limit of sums of the form $a_n:=\sum_{i=0}^n\lambda_{i,n}r_{i,n}$, where
$\lambda _{i,n}\in spec(a), n\in\Nat$, $r_{i,n}\in P\cap CC(a)$.
Let $\Psi(a_n)=g_n=\sum_{i=1}^n \lambda_{i,n}\Psi(r_{i,n})\in C(X,{\mathbb R})$ and $\Psi(a)=g\in C(X,{\mathbb R})$.
Then $f(a_n)= \Psi^{-1}(f\circ g_n)=\sum_{i=1}^n f(\lambda_{i,n})r_{i,n}$ with $f(\lambda_{i,n})\in spec(f(a))=spec(f\circ g)$.
We have $\|f(a_n)-f(a)\|= \|f\circ g_n-f\circ g\|$. Since $f$ is continuous on the compact set $spec(a)=spec(g)$,
it is uniformly continuous, and from $\|a_n-a\|=\|g_n-g\|\rightarrow 0$ we get $\|f(a_n)-f(a)\|\rightarrow 0$.

This proves that $f(a_n), n\in\Nat$ are Riemannian sums for $f(a)$, which entails
the result.
\end{proof}


\begin{thebibliography}{WW}
\bibitem{Alf} E.M. Alfsen: Compact Convex Sets and Boundary Integrals, Springer-Verlag,
Heidelberg, New York, 1971.

\bibitem{BW} G. Barbieri, H. Weber: Measures on clans and on MV-algebras, in: E. Pap (
Ed.), Handbook of Measure Theory, vol. II, Elsevier, Amsterdam, 2002, Chapter 22.

\bibitem{Birk} Garrett Birkhoff, Lattice Theory, A.M.S. Colloquium Publications,
XXV, Providence, R.I., 1967.

\bibitem{BLPM}  P. Busch, P. Lahti, J. J. Pekka, P. Mittelstaedt: The quantum
theory of measurement, Second edition. Lecture Notes in Physics. New Series m:
Monographs, 2. Springer-Verlag, Berlin, 1996.

\bibitem{BBGP} S. Gudder, E. Beltrametti, S. Bugajski,  S. Pulmannov\'{a}:
Convex and linear effect algebras, Rep. Math. Phys. {\bf 44} (1999) 359--379.

\bibitem{BGP} S. Bugajski, S. Gudder, S. Pulmannov\'a: Convex effect algebras,
state ordered effect algebras and ordered linear spaces, Rep. Math. Phys.
{\bf 45} (2005), 371--388.

\bibitem{BuKl} D. Butnariu, E.P. Klement: Triangular norm-based measures and
their Markov kernel representation, J. Math. Anal. Appl. {\bf 162} (1991) 111--143.

\bibitem{Chang} C.C. Chang: Algebraic analysis of many-valued logic, Trans. Amer.
Math. Soc. {\bf 88} (1958) 467--490.

\bibitem{Crist} R. Cristecsu: Ordered Vector Spaces and Linear Operators, Editura
Academii Bucuresti and  ABACUS Press, Tunbridge Wells, Kent, 1976.

\bibitem{Dvls}A. Dvure\v censkij: Loomis-Sikorski theorem for $\sigma$-complete
MV-algebras and $\ell$-groups, J. Austral. Math. Soc. {\bf Ser. A 68} (2000) 261-277.

\bibitem{DvPu} A. Dvure\v censkij, S. Pulmannov\'a: New Trends in Quantum Structures,
Inter Science, Bratislava and Kluwer, Dordrecht, 2000.

\bibitem{DvPucond} A. Dvure\v censkij, S. Pulmannov\'a: Conditional probability on
$\sigma$-MV-algebras, Fuzzy Sets and Syst. {\bf 155} (2005), 102--118.

\bibitem{FoBe} D.J. Foulis, M.K. Bennett: Effect algebras and unsharp quantum
logics, Found. Phys. {\bf 24} (1994) 1331--1352.

\bibitem{Fsa} D.J. Foulis: Synaptic algebras, Math. Slovaca {\bf 60} (2010)
631--654.

\bibitem{FPsr} D.J. Foulis, S. Pulmannov\'{a}:  Spectral resolution in an
order-unit space, Rep. Math. Phys. {\bf 62}, no. 3 (2008) 323-–344.

\bibitem{FPgh} D.J. Foulis, S. Pulmannov\'a: Generalized Hermitian algebras,
Internat. J. Theoret. Phys. {\bf 48} (2009) 1320--1333.

\bibitem{FPspin} D. Foulis, S. Pulmannov\'a: Spin factors as generalized Hermitian
algebras, Found. Phys. {\bf 39} (2009) 237--255.

\bibitem{FPproj} D.J. Foulis, S. Pulmannov\'a: Projections in a synaptic algebra,
Order {\bf 27} (2010) 235--257.

\bibitem{FPreg} D.J. Foulis, S. Pulmannov\'a: Regular elements in generalized
Hermitian algebra, Math. Slovaca {\bf 61} (2011) 155--172.

\bibitem{FPtype} D.J. Foulis, S. Pulmannov\'a: Type-decomposition of a synaptic
algebra, Found. Phys. {\bf 43}, no.  8  (2013) 948--968.

\bibitem{FPSym} D.J. Foulis, S. Pulmannov\'a:: Symmetries in a synaptic algebra,
Math. Slovaca {\bf 64}, no. 3 (2014) 751--776.

\bibitem{FPcom} D.J. Foulis, S. Pulmannov\'a: Commutativity in synaptic algebras,
Math. Slovaca {\bf 66} (2016) 469--482.

\bibitem{FPHand} D.J. Foulis, S. Pulmannov\'a: Handelman's theorem for an order
unit normed space, ArXiv:1609.08014v1[math.FA] 23 Sep 2016.

\bibitem{FJP2proj} D.J. Foulis, A. Jen\v cov\'a, S. Pulmannov\'a: Two projections
in a synaptic algebra, Linear Algebra Appl. {\bf 478} (2015) 163--287.

\bibitem{FJPprojef} D.J. Foulis, A. Jen\v cov\'a, S. Pulmannov\'a: A projection
and an effect in a synaptic algebra, Linear Algebra Appl. {\bf 485} (2015) 417--441.

\bibitem{FJPstat} D.J. Foulis, A. Jen\v cov\'a, S. Pulmannov\'a: States and
synaptic algebras, Rep. Math. Phys., to appear, ArXiv:1606.08229[math-ph].

\bibitem{FJPvect} D.J. Foulis, A. Jen\v cov\'a, S. Pulmannov\'a: Vector lattices
in synaptic algebras, to appear, ArXiv:1605.06987[math.RA].

\bibitem{Good} K.R. Goodearl: Partially Ordered Abelian Groups with Interpolation,
Math. Surveys and Monographs No. 20, AMs, Providence, 1986.

\bibitem{Hal} P. R. Halmos, Measure Theory, D. van Nostrand, New York, 1954.

\bibitem{Hand} D. Handelman: Rings with involution as partially ordered abelian
groups, Rocky Mt. J. Math. {\bf 11} (1981) 337--381.

\bibitem{K} G. Kalmbach: Orthomodular Lattices, Academic Press, London, New
York, 1983.

\bibitem{Loo} L.H. Loomis: On the representation of $\sigma$-complete Boolean
algebras, Bull. Amer. Math. Soc. {\bf 53} (1947) 757--760.

\bibitem{Mac} G.W. Mackey: Mathematical Foundations of Quantum Mechanics,
Benjamin, New York, 1963.

\bibitem{Mcc} K. McCrimmon: A taste of Jordan algebras, Universitext,
Springer-Verlag, New York, 2004.

\bibitem{Munls} D. Mundici: Tensor product and the Loomis-Sikorski
theorem for MV-algebras, Adv. Appl. Math. {\bf 22} (1999) 227--248.

\bibitem{Munint} D. Mundici: Interpretation of AF-C*-algebras in {\L}ukasiewicz
sentential calculus, J. Funct. Anal. {\bf 65} (1986) 15--63.

\bibitem{Pmv} S. Pulmannov\'a: A spectral theorem for $\sigma$-MV algebras,
Kybernetika {\bf 41} (2005), 361-374.

\bibitem{Pspec} S. Pulmannov\'a: Spectral resolutions in Dedekind
$\sigma$-complete $\ell$-groups, J. Math. Anal. Appl. {\bf 309} (2005) 322--335.

\bibitem{Pid} S. Pulmannov\'a: A note on ideals in synaptic algebras, Math.
Slovaca {\bf 62}, no. 6 (2012), 1091--1104.

\bibitem{PtPu} P. Pt\'ak, S. Pulmannov\'a: Orthomodular Structures as Quantum Logics,
Kluwer, Dordrecht, 1991.

\bibitem{ZR} Z. Rie\v{c}anov\'{a}, Generalization of blocks for D-lattices and
lattice-ordered effect algebras, Internat. J. Theoret. Phys. {\bf 39}, no. 2
(2000) 231-–237.

\bibitem{Sch} H.H. Schaefer: Banach Lattices and Positive operators, Springer,
Berlin Heidelberg New York, 1974.

\bibitem{Sik} R. Sikorski: Boolean Algebras, Springer, Berlin-Heidelberg-New
York, 1964.

\bibitem{Var} V.S. Varadarajan: Geometry of Quantum Theory, Springer-Verlag,
New York-Berlin, 1985.

\end{thebibliography}
\end{document}